\documentclass[12pt,reqno]{amsart}

\usepackage[T1]{fontenc}
\usepackage[utf8]{inputenc}
\usepackage[a4paper,margin=1.15in]{geometry}
\usepackage[osf]{mathpazo}       
\usepackage{microtype}

\usepackage{amsmath,amssymb,amsthm,mathtools}
\usepackage{enumitem}
\usepackage{xcolor}
\definecolor{RoyalBlue}{RGB}{30,70,150}
\usepackage[colorlinks=true,
            linkcolor=RoyalBlue,
            citecolor=RoyalBlue,
            urlcolor=RoyalBlue]{hyperref}
\usepackage[nameinlink,capitalise]{cleveref}
\usepackage{comment}
\usepackage{tikz}
\usetikzlibrary{matrix,positioning,arrows.meta,fit,backgrounds,calc}

\setlist[itemize]{leftmargin=1.8em,itemsep=0.3em,topsep=0.3em}
\setlist[enumerate]{leftmargin=1.9em,itemsep=0.3em,topsep=0.3em}

\newtheorem{theorem}{Theorem}[section]
\newtheorem{proposition}[theorem]{Proposition}
\newtheorem{lemma}[theorem]{Lemma}
\newtheorem{corollary}[theorem]{Corollary}
\newtheorem{claim}[theorem]{Claim}

\theoremstyle{definition}

\theoremstyle{remark}

\newcommand{\R}{\mathbb{R}}

\newcommand{\Q}{\mathbb{Q}}
\newcommand{\Z}{\mathbb{Z}}
\newcommand{\F}{\mathbb{F}}
\newcommand{\OK}{\mathcal{O}_K}
\newcommand{\Norm}{\mathrm{N}}
\newcommand{\Disc}{\Delta}
\newcommand{\eps}{\varepsilon}
\newcommand{\fp}{\mathfrak{p}}

\crefname{equation}{}{}
\Crefname{equation}{}{}
\crefname{theorem}{Theorem}{Theorems}
\Crefname{theorem}{Theorem}{Theorems}
\crefname{proposition}{Proposition}{Propositions}
\Crefname{proposition}{Proposition}{Propositions}
\crefname{lemma}{Lemma}{Lemmas}
\Crefname{lemma}{Lemma}{Lemmas}
\crefname{remark}{Remark}{Remarks}
\Crefname{remark}{Remark}{Remarks}

\title[Split primes and the Elekes-R\'onyai problem]
      {Split primes and the Elekes-R\'onyai problem}
\author{Cosmin Pohoata}
\thanks{Department of Mathematics, Emory University, Atlanta, USA. Research supported by NSF grant DMS-2246659. Email: \texttt{cosmin.pohoata@emory.edu}}
\date{}

\begin{document}

\begin{abstract}
There exist an absolute constant $c>0$ and arbitrarily large finite sets $A\subset\R$
with
\[
    \left| \left\{x+y+(x-y)^2:\ x, y \in A\right\}\right| \le|A|^{2-c}.
\]
Since $x+y+(x-y)^2 \in \mathbb{R}[x,y]$ is a polynomial which is neither additive nor multiplicative, this provides a counterexample for the Elekes-R\'onyai problem. 

The proof combines two amplifications of the same local congruence defect: \emph{horizontal} amplification over squarefree products of rational primes, and \emph{vertical} amplification through bounded root-discriminant towers in which those primes split completely. In this way a fixed local density defect becomes macroscopic, producing a power saving. This phenomenon also suggests a broader mechanism for producing similar extremal constructions throughout combinatorics and number theory.
\end{abstract}

\maketitle

\section{Introduction} \label{sec:intro}

Let $f\in\R[x,y]$ be a fixed bivariate polynomial and let $A,B\subset\R$ be
finite sets, each of size $n$. The Elekes--R\'onyai problem asks how small the
image set
\[
        f(A,B)=\{\,f(a,b):a\in A,\ b\in B\,\}
\]
can be. Two families of polynomials evade expansion for trivial reasons. If $f$ is additive, in the sense that $f(x,y)=g\bigl(u(x)+v(y)\bigr)$ for univariate $g,u,v$, one may take $u(A)$ and $v(B)$ to be arithmetic progressions. If $f$ is multiplicative, in the sense that $f(x,y)=g\bigl(u(x)\,v(y)\bigr)$, then one can take geometric progressions instead. In either case $|f(A,B)|$ can be kept linear in $n$. For convenience, throughout this paper, we will sometimes refer collectively to the polynomials which are either additive or multiplicative as the {\it{special forms}}, and to polynomials which are neither additive nor multiplicative as {\it{non-special}}.

In a highly influential paper from 2000, Elekes and R\'onyai \cite{ElekesRonyai2000} proved that the additive and multiplicative special forms above are the only ways to possibly have linear-size image in the following qualitative sense: for every fixed non-special polynomial $f \in \mathbb{R}[x,y]$, the minimum possible image size of an $n\times n$ Cartesian product is always superlinear in $n$:
\[
        \min_{\substack{A,B\subset\R\\ |A|=|B|=n}}
        \frac{|f(A,B)|}{n}\longrightarrow\infty
        \qquad\text{as } n\to\infty.
\]
Equivalently, if $|f(A,B)|=O_f(n)$ for arbitrarily large Cartesian products,
then $f$ must be additive or multiplicative~\cite{ElekesRonyai2000}. Establishing quantitative versions of this phenomenon and related variants has since become a central topic at the intersection of incidence geometry and arithmetic combinatorics. In particular, a well-known conjecture of Elekes states that this qualitative superlinearity should in fact be nearly quadratic: for every $\eps>0$, there exists a positive constant $C_{f,\varepsilon} > 0$ such that every non-special polynomial satisfies
\begin{equation} \label{ER}
        |f(A,B)| \geq C_{f,\varepsilon} n^{2-\eps}\ \ \ \text{for every}\ A, B \subset \mathbb{R}\ \text{with}\ |A|=|B|=n.
\end{equation}
This conjecture is often referred to as the Elekes--R\'onyai problem. For historical background and more context, see for example the earlier paper Elekes \cite{ElekesDistinctDistances1999}, Section~4.1 of Matou\v{s}ek's book \cite{Matousek2002}, the introduction of Raz--Sharir--Solymosi
\cite[Section~1.1]{RazSharirSolymosi2016}, or the excellent survey by de Zeeuw \cite{deZeeuwSurvey}. The best known quantitative estimate towards \eqref{ER} comes from the recent work of Solymosi and Zahl \cite{SolymosiZahl2024}, who proved that for every non-special polynomial $f \in \mathbb{R}[x,y]$, one must always have that $|f(A,B)| = \Omega_{f}(n^{3/2})$, for every two sets $A,B \subset \mathbb{R}$ with $|A|=|B|=n$. This in turn improved upon the previous record of $n^{4/3}$ from \cite{RazSharirSolymosi2016}.

In this paper, we disprove the conjecture of Elekes, by establishing the following result.

\begin{theorem}
\label{thm:main}
There exist arbitrarily large finite sets $A\subset\R$ and an absolute constant
$c>0$ such that
\[
        |f(A,A)|\le |A|^{2-c},
\]
where
\[
        f(x,y)=x+y+(x-y)^2.
\]
Moreover, for every fixed $\eps>0$, the sets may be chosen so that
$|A+A|\le |A|^{1+\eps}$ while $|f(A,A)|\le |A|^{2-c_\eps}$ for some
$c_\eps>0$.
\end{theorem}

Since $f$ is not of additive or multiplicative Elekes--R\'onyai special form, the first part of Theorem \ref{thm:main} provides a counterexample to \eqref{ER}. 

The construction was first announced in the blog post ~\cite{PohoataERBlog} by the author on June 1st, 2026. The present paper supplies the formal presentation, along with more details and additional context. 

\medskip

\noindent\textbf{Proof overview.}
The main idea behind the proof of Theorem~\ref{thm:main} is inspired a new
combinatorial large sieve method recently developed by Croot, Mao, the author, Sheffer, and Yip in \cite{CMPSY2026}. Roughly speaking,
the results from \cite{CMPSY2026} are driven by a common local-to-global
mechanism: an algebraic congruence splits into many compatible branches modulo
many small primes, and a global bounded-multiplicity hypothesis prevents all of
these local coincidences from occurring too often.  Here we use the same local
principle in the opposite direction.  Instead of using local restrictions to
prove that a set with bounded multiplicity must be small, we build a polynomial
whose values are forced into a small collection of local residue classes.

For a number field $K$, with ring of integers $\mathcal O_K$, let $Q=p_1\cdots p_t$ be a squarefree product of odd (rational) primes $p_1,\ldots,p_t$ which split completely in $K$. Consider the polynomial
\begin{equation}
\label{eq:fQ-intro}
        f_Q(x,y)=Q(x+y)+(x-y)^2\in \mathbb Z[x,y],
\end{equation}
evaluated on algebraic integers $x,y\in \mathcal O_K$. For each prime $p_i\mid Q$, complete splitting means that, if $d=[K:\Q]$, then we have the factorization $p_i\mathcal O_K=\mathfrak p_{i,1}\cdots \mathfrak p_{i,d}$ into prime ideals $\mathfrak p_{i,j} \in \mathrm{Spec}(\mathcal O_K)$, for $j = 1,\ldots d$. Each prime ideal determines an independent residue field $\mathcal O_K/\mathfrak p_{i,j}\cong \mathbb F_{p_i}$. 

Crucially, modulo any prime ideal $\mathfrak p_{i,j}\subset \mathcal O_K$, the term
$Q(x+y)$ vanishes, and therefore $f_Q(x,y)\equiv (x-y)^2 \pmod{\mathfrak p_{i,j}}$. Thus, in each residue-field coordinate, the value of $f_Q$ is forced to lie
among the quadratic residues, including zero, a set of size $(p_i+1)/2$ in
$\mathbb F_{p_i}$. Since the ideals \(\mathfrak p_{i,1},\ldots,\mathfrak p_{i,d}\) are pairwise comaximal, the Chinese remainder theorem gives
\[
        \mathcal O_K/Q\mathcal O_K
        \cong
        \prod_{i=1}^t\prod_{j=1}^d
        \mathcal O_K/\mathfrak p_{i,j}
        \cong
        \prod_{i=1}^t \mathbb F_{p_i}^{\,d}.
\]
Consequently, it follows that the image of $f_Q$ modulo $Q\mathcal O_K$ is confined to a subset of residue density at most
\[
        \left(
        \frac{1}{Q}\prod_{i=1}^t \frac{p_i+1}{2}
        \right)^d
        = \prod_{i=1}^t
        \left(\frac{p_i+1}{2p_i}\right)^d.
\]
Since $p_i + 1 < 2p_i$ holds for each $i=1,\ldots,t$, the latter is of the form $\rho^{d}$ for some $0<\rho <1$. This engineered residue-class bottleneck is the entire local mechanism behind the construction.

For a number field \(K\), let \(Q=p_1\cdots p_t\) be a product of rational
primes which split completely in \(K\), and consider
\begin{equation}
\label{eq:fQ-intro}
        f_Q(x,y)=Q(x+y)+(x-y)^2.
\end{equation}
If \(d=[K:\Q]\), complete splitting supplies \(d\) independent
Chinese remainder theorem coordinates above each \(p_i\).  In every such coordinate,
\[
        f_Q(x,y)\equiv(x-y)^2,
\]
so the value is forced to be a square.  Consequently, the image modulo
\(Q\OK\) occupies at most a \(\rho^d\)-proportion of the residue classes, where
\[
        \rho=\prod_{i=1}^t\frac{p_i+1}{2p_i}<1.
\]
This exponential-in-degree residue-class bottleneck is the local mechanism
behind the construction.

To convert this into a counterexample for the Elekes--R\'onyai problem, we will further require some ingredients from two other recent breakthrough constructions. The first one is the striking counterexample to the Erd\H{o}s unit-distance conjecture, exhibited last month by OpenAI \cite{OpenAIUnitDistanceBlog,UnitDistanceRemarks}. The second is the subsequent work of Bloom, Sawin, Schildkraut, and Zhelezov, which, inspired by the ideas from \cite{OpenAIUnitDistanceBlog}, also managed to provide a beautiful counterexample to the sum--product conjecture~\cite{BSSZ2026} (over the reals, as well as several other well-studied variants).  

Specifically, we will take advantage of high-dimensional symmetric Minkowski boxes in rings of integers, in the same style as \cite{BSSZ2026}. These boxes also turn out to have small additive doubling. Second, we use the bounded root-discriminant tower of totally real fields together with the infinite supply of primes splitting completely in every layer, which was also the main arithmetic input behind the unit-distance construction from ~\cite{OpenAIUnitDistanceBlog}. The point in our case will be that we can use these primes to fix $Q$ once and for all, while allowing the degree $d=[K:\mathbb Q]$ to tend to infinity. The local square-class restriction accumulates as a factor $\rho^d$, with $\rho<1$, and, at the same time, bounded root discriminant gives additive boxes of size $\exp(\Theta(d))$. Thus the exponential-in-$d$ congruence saving becomes a fixed power saving in $|A|$.

\medskip

{\noindent \bf{Paper organization.}} In \Cref{sec:warm-up} we give a simpler model of the construction over $\Z$, which only gives a subpolynomial saving but displays the local-to-global congruence mechanism in its simplest form (without appeal to any algebraic number theory input or terminology). \Cref{sec:number-field} records the split-prime tower input and the two elementary lattice estimates that drive the counting argument. \Cref{sec:construction} carries out the construction for fixed $f_Q$, verifies non-specialness, and rescales to the fixed polynomial of \cref{thm:main}. Finally, in \Cref{sec:amplification} we conclude with a more in-depth discussion of the general mechanism behind the proof of Theorem \ref{thm:main}, with an eye on further applications.

\section{A finite interval warm-up}
\label{sec:warm-up}

Let $[N]=\{1,2,\dots,N\}$. Here the modulus $Q$ will depend on $N$, so this warm-up does not by itself prove \cref{thm:main}, where the polynomial must be fixed. It does, however, show
exactly why the polynomial $f_Q$ was chosen: modulo every prime divisor of $Q$,
its values are forced into the set of quadratic residues.

\begin{proposition}
\label{prop:interval-warmup}
Let $p_1,\dots,p_t$ be distinct odd primes, let $Q=p_1\cdots p_t$, and let $\rho=\frac{1}{Q} \prod_{i=1}^{t}\frac{p_i+1}{2}$. Then, for every $N\ge1$,
$$|f_Q([N],[N])| \le \rho\,(2QN+N^2+Q).$$
\end{proposition}

\begin{proof}
For $a,b\in[N]$, the value $f_Q(a,b)=Q(a+b)+(a-b)^2$ is an integer. Moreover, $2Q\le f_Q(a,b)\le 2QN+(N-1)^2$, so all values lie in an interval of length at most $2QN+N^2$. Now fix $p_i\mid Q$. Since $Q(a+b)\equiv0\pmod{p_i}$, we have $f_Q(a,b)\equiv (a-b)^2\pmod{p_i}$. Thus modulo $p_i$ every value of $f_Q(a,b)$ lies in the set of squares in
$\mathbb F_{p_i}$, which has size $(p_i+1)/2$. By the Chinese remainder theorem,
the values of $f_Q(a,b)$ are therefore confined to at most
\[
        \prod_{i=1}^{t}\frac{p_i+1}{2}
        =
        \rho Q
\]
residue classes modulo $Q$.

If an interval has length $L$, then each residue class modulo $Q$ contributes at
most $L/Q+1$ integers to that interval. Applying this with
$L=2QN+N^2$ and with the allowed set of residues above gives
\[
        |f_Q([N],[N])|
        \le
        \rho Q\left(\frac{2QN+N^2}{Q}+1\right)
        =
        \rho(2QN+N^2+Q),
\]
as claimed.
\end{proof}

For arbitrarily large $N$, one may use the Prime Number Theorem to choose distinct primes
$p_1,\dots,p_t\in[\log N,2\log N]$ with $t=(1+o(1))\frac{\log N}{\log\log N}$, and then consider $Q=p_1\cdots p_t\le (2 \log N)^{t} \leq N$. We have
\[
        \prod_{i=1}^{t}\left(1+\frac1{p_i}\right)
        \leq \exp\!\left(\sum_{i=1}^{t}\frac1{p_i}\right)
        \leq 
\exp\!\left(\frac{t}{\log N}\right) = \exp(o(1)),
\]
therefore 
\[
        \rho
        =
        2^{-t}\prod_{i=1}^{t}\left(1+\frac1{p_i}\right)
        =
        \exp\!\left(-(\log 2)t+o(1)\right)
        =
        \exp\!\left(-\bigl(\log 2-o(1)\bigr)
        \frac{\log N}{\log\log N}\right).
\]

Since $Q\le N$, Proposition \ref{prop:interval-warmup} thereby implies the following

\begin{corollary}
\label{cor:interval-cmy}
For the choice of $Q$ above,
$$|f_Q([N],[N])| \le N^2\exp\!\left(-\bigl(\log 2-o(1)\bigr) \frac{\log N}{\log\log N}\right).$$
\end{corollary}
In particular, if $f(x,y)=x+y+(x-y)^2$, then for any $a,b\in[N]$ we have
\[
        f_Q(a,b)=Q^2 f(a/Q,b/Q).
\]
Thus, taking $A=Q^{-1}[N]$, multiplication by the nonzero scalar $Q^{-2}$
identifies $f_Q([N],[N])$ with $f(A,A)$, so the two image sets have the same cardinality. 

For the main construction below we will adapt this strategy as follows.  Instead of letting $Q$ grow with $N$, we fix $Q$ once and for all, and create the extra independent coordinates from a different source: rather than new primes, we use
new prime \emph{ideals}. Passing to a number field $K$ of degree $d$ in which each $p_i$ splits completely replaces the single coordinate at $p_i$ by $d$ of them, one for each prime ideal above $p_i$, each imposing the same square-class restriction. Since it will possible to pick the number field $K$ with arbitrarily large degree $d$, while maintaining the root discriminant of $K$ uniformly bounded, then this will have the effect of amplifying the fixed local saving into a genuine power of the set size. 

We first record the main number theoretic input and the elementary geometry-of-numbers estimates that will allow for all this to happen in the next section. 

\section{The arithmetic input}
\label{sec:number-field}

The star of the show is the following result from Hajir, Maire, and Ramakrishna~\cite{HajirMaireRamakrishna2021}, which is also the main driving force behind the recent unit distance construction from \cite{OpenAIUnitDistanceBlog}. This result also appears as Proposition~2.3 in \cite{UnitDistanceRemarks}. The
underlying bounded root-discriminant tower technology goes back to the work of
Martinet~\cite{Martinet1978} and Hajir--Maire~\cite{HajirMaire2001}.

\begin{proposition}
\label{prop:tower}
There exist totally real number fields $K_i$ with degrees
$d_i=[K_i:\Q]\to\infty$, an absolute constant $D>0$, and an infinite set
$\mathcal P$ of odd rational primes, such that
\[
        \Disc_{K_i}^{\,1/d_i}\le D
        \quad\text{for every }i,
        \qquad\text{and every }p\in\mathcal P
        \text{ splits completely in every }K_i.
\]
\end{proposition}

Here $\Disc_K$ denotes the absolute value of the discriminant of the number field $K$.  The
condition $\Disc_K^{1/[K:\Q]}\le D$ says that the root discriminant is bounded
uniformly along the tower.  This keeps the Minkowski lattice of $\OK$ from becoming too sparse as the degree grows. 

For every number field $K$ in the tower from Proposition \ref{prop:tower}, recall that a prime $p \in \mathcal P$ completely splits in $K$ if there exist distinct prime ideals $\fp_1,\ldots \fp_d$ in the ring of integers $\OK$, such that $p\OK=\fp_1\cdots\fp_d$, where $d=[K:\Q]$ denotes the degree of $K$ over $\mathbb{Q}$. The important point is that these primes $p \in \mathcal P$ split in {\it{every}} layer of the tower.

For the rest of the section, fix a totally real field $K$ of degree $d$, with real
embeddings $\sigma_1,\ldots,\sigma_d:K\hookrightarrow\R$. For $X\ge1$, define the symmetric Minkowski box of radius $X$ to be the set
\[
        B_K(X)=
        \{\alpha\in\OK:\ |\sigma_i(\alpha)|\le X
        \text{ for all }1\le i\le d\}.
\]
Under the Minkowski embedding
\[
        \alpha\longmapsto
        (\sigma_1(\alpha),\ldots,\sigma_d(\alpha)) \subset \mathbb{R}^{d},
\]
the ring of integers $\OK$ is a full-rank lattice in $\R^d$ of covolume $\Disc_K^{1/2}$; see, for example, \cite[Chapter V, Lemma 2]{LangANT}. 

We write $\Norm_{K/\mathbb Q}$ for the field norm.  Thus, for
$\alpha\in K$,
\[
        \Norm_{K/\mathbb Q}(\alpha)
        =
        \prod_{\sigma:K\hookrightarrow\mathbb C}\sigma(\alpha).
\]
Since $K$ is totally real, its embeddings are precisely
$\sigma_1,\ldots,\sigma_d:K\hookrightarrow\mathbb R$, and so in the present
setting this becomes $\Norm_{K/\mathbb Q}(\alpha) = \prod_{i=1}^d \sigma_i(\alpha)$. Equivalently, $\Norm_{K/\mathbb Q}(\alpha)$ is the determinant of the
$\mathbb Q$-linear multiplication map $m_\alpha:K\to K$ defined by $m_{\alpha}(z) = \alpha z$. So, for example, if $\alpha\in\mathcal O_K$, then
$\Norm_{K/\mathbb Q}(\alpha)\in\mathbb Z$, because multiplication by $\alpha$
preserves the lattice $\mathcal O_K$.  In particular, a nonzero algebraic integer has nonzero integral norm, and therefore $|\Norm_{K/\mathbb Q}(\alpha)|\ge1$. 

We shall use the following elementary separation observation twice.  If
$\alpha,\beta\in\OK$ are distinct and $\alpha\equiv\beta\pmod{R\OK}$ for some rational integer $R\ge1$, then $\alpha-\beta=R\gamma$ for a nonzero $\gamma\in\OK$.  Since $\gamma$ is a nonzero algebraic integer, by the discussion above we have that
\[
        1\le |\Norm_{K/\Q}(\gamma)| = \prod_{i=1}^d |\sigma_i(\gamma)|.
\]
Hence $|\sigma_i(\gamma)|\ge1$ for at least one embedding $i$, and therefore $|\sigma_i(\alpha-\beta)|\ge R$. Thus two distinct algebraic integers in the same residue class modulo $R\OK$ are $R$-separated in the $\ell^\infty$ metric after the Minkowski embedding. We are now ready to record our geometry-of-number estimates. 

The first result will allow us to establish the additive-box estimate, which we advertised already at the end of Section \ref{sec:intro}.

\begin{lemma}
\label{lem:box-count}
For every totally real field $K$ of degree $d$ and every $X\ge1$,
\[
        X^{d}\Disc_K^{-1/2}
        \le
        |B_K(X)|
        \le
        (2X+1)^{d}.
\]
\end{lemma}

This is precisely \cite[Lemma 3.3]{BSSZ2026}, however, for the reader's convenience, let us include the short proof here as well. The more interesting part is the lower bound, which is an application of Blichfeldt's lemma from \cite{{Blichfeldt1914}}. We use
Blichfeldt's lemma in the standard lattice form: if $\Lambda\subset\R^d$ is a lattice of covolume $\det\Lambda$ and $S\subset\R^d$ is bounded and measurable, then some translate $z+S$ contains at least $\operatorname{vol}(S)/\det\Lambda$ points of $\Lambda$.  See, for instance,
Cassels~\cite[Chapter~III, Section~2]{Cassels1997} and the references therein.

\begin{proof}[Proof of Lemma \ref{lem:box-count}]
Let $\Lambda$ be the Minkowski lattice of $\OK$ in $\R^d$. For the upper bound, apply the separation observation with $R=1$.  Distinct points of $\Lambda\cap[-X,X]^d$ are $1$-separated in $\ell^\infty$.  Placing
half-open cubes of side length $1$ centered at these points gives disjoint cubes,
all contained in $[-X-\tfrac12,X+\tfrac12]^d$. Comparing volumes gives $|B_K(X)|\le (2X+1)^d$.

For the lower bound, apply Blichfeldt's lemma in the form recalled above to $S=[-X/2,X/2]^d$. Since $\operatorname{vol}(S)=X^d$ and the covolume of the Minkowski lattice is
$\Disc_K^{1/2}$, some translate $z+S$ contains at least
$X^d\Disc_K^{-1/2}$ lattice points.  Choose one of them, say $v_0$, and
subtract it from all the others.  If $v,v_0\in z+S$, then $v-v_0\in S-S=[-X,X]^d$. Because the Minkowski lattice is an additive group, these differences are again
lattice points; translating back, they are distinct algebraic integers whose
conjugates all have absolute value at most $X$.  Hence they lie in $B_K(X)$,
and $|B_K(X)|\ge X^d\Disc_K^{-1/2}$.
\end{proof}

The second estimate is the same packing argument from the upper bound in Lemma \ref{lem:box-count}, but carried out separately in each residue class. We record this separately mostly for reference purposes. 

\begin{lemma}
\label{lem:residue-count}
Let $R\ge1$ be a rational integer, let $\Omega\subseteq\OK/R\OK$ be a set of
residue classes, and let $M\ge1$.  Then
\[
        \#\{\alpha\in\OK:
        |\sigma_i(\alpha)|\le M\ \ \text{for every}\ i,\ \  
        \alpha\bmod R\OK\in\Omega\}
        \le
        |\Omega|\left(\frac{2M}{R}+1\right)^d.
\]
\end{lemma}

\begin{proof}[Proof of Lemma \ref{lem:residue-count}]
It is enough to count points in one residue class modulo $R\OK$.  By the separation observation above, distinct points in a fixed residue class are $R$-separated in the $\ell^\infty$ metric after the Minkowski embedding. Now restrict to the cube $[-M,M]^d$.  Around each point in the fixed residue
class place a half-open cube of side length $R$ centered at that point.  These
small cubes are disjoint, and they all lie in the enlarged cube $[-M-R/2,M+R/2]^d$. Therefore the number of points in this residue class is at most
\[
        \frac{(2M+R)^d}{R^d}
        =
        \left(\frac{2M}{R}+1\right)^d.
\]
Multiplying by the number $|\Omega|$ of allowed residue classes proves the lemma.
\end{proof}

\section{Proof of Theorem \ref{thm:main}}
\label{sec:construction}

Select distinct primes $p_1,\dots,p_t$ from the split set $\mathcal P$ of Proposition~\ref{prop:tower}, set $Q=p_1\cdots p_t$, and let $f_Q$ be the polynomial
$$f_Q(x,y)=Q(x+y)+(x-y)^2\in \mathbb Z[x,y],$$
as in \eqref{eq:fQ-intro}. Write
\[
        \rho=\prod_{i=1}^{t}\frac{p_i+1}{2p_i}
\]
for the product of the local densities of the squares in the residue field $\mathbb F_{p_i}$. Since each $p_i$ is odd, each factor is at most $2/3$,
and hence $\rho$ can be made arbitrarily small by taking enough primes from
$\mathcal P$.  We choose $p_1,\ldots,p_t$ once and for all so that $\theta:=13D\rho<1$, where $D$ is the root-discriminant constant from Proposition~\ref{prop:tower}. From this point on, $Q$, $\rho$, and $\theta$ are fixed. We isolate the following main claim. 

\begin{proposition}
\label{prop:fQ}
For the squarefree integer $Q$ chosen above, there exist an absolute constant $c>0$ and
arbitrarily large finite sets $A_0\subset\mathbb R$ such that
\[
        |f_Q(A_0,A_0)|\le |A_0|^{2-c}.
\]
Moreover, for every fixed $\varepsilon>0$, the sets may be chosen so that $|A_0+A_0|\le |A_0|^{1+\varepsilon}$ and still $|f_Q(A_0,A_0)|\le |A_0|^{2-c_\varepsilon}$ for some constant $c_\varepsilon>0$.
\end{proposition}

\begin{proof}
Let $K$ be a field in the tower of Proposition~\ref{prop:tower}, of degree
$d=[K:\mathbb Q]$, so that $\Disc_K^{1/d}\le D$. Choose a real embedding $\sigma_1:K\hookrightarrow\mathbb R$, fix a real
number $X\ge \max\{Q,2D^{1/2},1\}$, and let $P=B_K(X)$ be the symmetric Minkowski box of radius $X$. Finally, set $A_0=\sigma_1(P)\subset\mathbb R$.

Since $\sigma_1$ is one-to-one on $K$, we have $n:=|A_0|=|P|$. By Lemma~\ref{lem:box-count},
\begin{equation}
\label{eq:sizeP}
        \left(\frac{X}{D^{1/2}}\right)^d
        \le n
        \le (2X+1)^d.
\end{equation}
In particular $n\to\infty$ as $d\to\infty$ along the tower.  Thus the sets
constructed below have arbitrarily large cardinality.

We now estimate $f_Q(P,P)\subset\mathcal O_K$. First, if $a,b\in P$, then for every
embedding $\sigma_i$,
\[
|\sigma_i(f_Q(a,b))| \le Q|\sigma_i(a+b)|+|\sigma_i(a-b)|^2 \le 2QX+4X^2 \le 6X^2,
\]
where we used $X\ge Q$.  Hence we have the inclusion
\begin{equation}
\label{eq:archimedean-fQ}
        f_Q(P,P)\subseteq B_K(6X^2).
\end{equation}

Second, recall that the values of $f_{Q}$ occupy few residue classes modulo $Q\mathcal O_K$. We execute the idea from the end of Section \ref{sec:intro}: each prime $p_i$ splits completely in $K$, so we may write
\[
        p_i\mathcal O_K=\prod_{j=1}^d \mathfrak p_{i,j},
        \qquad \text{where}\ \
        \mathcal O_K/\mathfrak p_{i,j}\cong\mathbb F_{p_i}.
\]
For $a,b\in P$, reduction modulo $\mathfrak p_{i,j}$ kills the linear term, due to the fact that $p_i\mid Q$, i.e.
\[
        f_Q(a,b)
        =
        Q(a+b)+(a-b)^2
        \equiv
        (a-b)^2
        \pmod{\mathfrak p_{i,j}}.
\]
In particular, in each residue field $\mathcal O_K/\mathfrak p_{i,j}\cong\mathbb F_{p_i}$,
the value of $f_Q(a,b)$ is forced to be a square.  The set of squares in
$\mathbb F_{p_i}$, including $0$, has size $(p_i+1)/2$.  By the Chinese
remainder theorem,
\[
        \mathcal O_K/Q\mathcal O_K
        \cong
        \prod_{i=1}^t\prod_{j=1}^d
        \mathcal O_K/\mathfrak p_{i,j}
        \cong
        \prod_{i=1}^t \mathbb F_{p_i}^{\,d}.
\]
Consequently the residues of $f_Q(a,b)$ modulo $Q\mathcal O_K$ are confined
to a set $\Omega\subseteq\mathcal O_K/Q\mathcal O_K$ of size at most
\begin{equation}
\label{eq:Omega-size}
        |\Omega|
        \le
        \prod_{i=1}^t
        \left(\frac{p_i+1}{2}\right)^d
        =
        (\rho Q)^d.
\end{equation}
Combining \eqref{eq:archimedean-fQ}, \eqref{eq:Omega-size}, and Lemma~\ref{lem:residue-count} with
$R=Q$ and $M=6X^2$, we get
\[
|f_Q(P,P)| \le |\Omega| \left(\frac{12X^2}{Q}+1\right)^d\le (\rho Q)^d \left(\frac{12X^2}{Q}+1\right)^d = \bigl(\rho(12X^2+Q)\bigr)^d.
\]
Since $Q \leq X^2$, the latter is $\le (13\rho X^2)^d$. On the other hand, the lower bound from \eqref{eq:sizeP} gives $n^2\ge \left(\frac{X^2}{D}\right)^d$, or equivalently $(X^2)^d\le D^d n^2$. Hence, using the definition of $\theta$, we can conclude that
\begin{equation}
\label{eq:theta-bound-main}
        |f_Q(P,P)|
        \le
        (13D\rho)^d n^2
        =
        \theta^d n^2.
\end{equation}
This is the key inequality: the residue restrictions give an exponentially small
factor $\theta^d$, in front of the trivial bound $n^2$. Since $\theta<1$, and since $n\le (2X+1)^d$, we have $\theta^d \le n^{-c}$, where $c=\frac{-\log\theta}{\log(2X+1)}>0$. Therefore $|f_Q(P,P)|\le n^{2-c}$. Finally, because $f_Q$ has rational coefficients and $\sigma_1$ is injective,
the map $\sigma_1$ identifies $f_Q(P,P)$ with $f_Q(A_0,A_0)$.  Thus $|f_Q(A_0,A_0)|=|f_Q(P,P)|\le |A_0|^{2-c}$.

It remains to record the small-doubling refinement.  Since $P+P\subseteq B_K(2X)$, Lemma~\ref{lem:box-count} gives
\[
        |A_0+A_0|
        =
        |P+P|
        \le
        |B_K(2X)|
        \le
        (4X+1)^d.
\]
Given $\varepsilon>0$, choose $X$, after $Q$ has been fixed, so large that $4X+1 \le \left(\frac{X}{D^{1/2}}\right)^{1+\varepsilon}$. This is possible because the left-hand side grows linearly in $X$, while the
right-hand side grows like $X^{1+\varepsilon}$.  Then the lower bound in
\eqref{eq:sizeP} gives $|A_0+A_0| \le |A_0|^{1+\varepsilon}$, as claimed. Increasing $X$ may decrease the exponent $c$, but once $X$ is fixed the
exponent remains a positive constant $c_\varepsilon>0$.
\end{proof}

\subsection*{Removing the coefficient \texorpdfstring{$Q$}{Q}}

After Corollary \ref{cor:interval-cmy}, we already saw that the auxiliary coefficient $Q$ is only a device for imposing congruence
restrictions. Here too it immediately disappears by considering the same simple scaling. 

Let $f(x,y)=x+y+(x-y)^2$. Then
\begin{equation}
\label{eq:scaling-fQ}
        f_Q(Qx,Qy)
        =
        Q^2\bigl(x+y+(x-y)^2\bigr)
        =
        Q^2 f(x,y).
\end{equation}
Therefore, if $A_0\subset\mathbb R$ is any finite set and $A=Q^{-1}A_0$, then $f(A,A)=Q^{-2}f_Q(A_0,A_0)$. Multiplication by the nonzero scalar $Q^{-2}$ does not change cardinality, so
\begin{equation}
\label{eq:scaling-cardinality}
        |f(A,A)|=|f_Q(A_0,A_0)|.
\end{equation}
Also $|A|=|A_0|$, and additive doubling is preserved: $|A+A|=|A_0+A_0|$. Thus Proposition~\ref{prop:fQ} immediately gives arbitrarily large sets
$A\subset\mathbb R$ such that $|f(A,A)|\le |A|^{2-c}$, with the same small-doubling refinement.  

\subsection*{Non-special verification} Last but not least, we include a routine verification that the fixed polynomial $f$ is genuinely non-special. We isolate this as a claim below.

\begin{claim}
\label{lem:non-special}
The polynomial
\[
        f(x,y)=x+y+(x-y)^2 \in \mathbb{R}[x,y]
\]
is neither additive nor multiplicative.
\end{claim}
In fact, no such representation exists even over $\mathbb C$:
there do not exist $a,b,c\in\mathbb C[t]$ such that $f(x,y)=a\bigl(b(x)+c(y)\bigr)$ or $f(x,y)=a\bigl(b(x)c(y)\bigr)$. The point is that $f$ has a deliberately incompatible pair of fingerprints: its quadratic part points in the $x-y$ direction, while its linear part points in the transverse $x+y$ direction. We record a formal proof below, supplied by ChatGPT. 

\begin{proof}
Suppose that $f(x,y)=a\bigl(b(x)+c(y)\bigr)$ holds for some $a,b,c\in\mathbb C[t]$.  Since $f$ depends genuinely on both
variables, both $b$ and $c$ are nonconstant.  If $a$ were linear, then the
right-hand side would be a sum of a function of $x$ and a function of $y$, and
so would have no $xy$-term.  This is impossible, since $f(x,y)=x+y+x^2-2xy+y^2$ has a nonzero mixed term.  Thus $a$ is nonlinear. 

Because $\deg f=2$, the only remaining possibility is to have $\deg a=2$ and $\deg b=\deg c=1$. Write $a(t)=\alpha t^2+\beta t+\gamma$, with $\alpha\ne0$,
and $b(x)+c(y)=ux+vy+w$ with $u,v\ne0$. The quadratic homogeneous part of $a(b(x)+c(y))$ is then $\alpha(ux+vy)^2$. But the quadratic homogeneous part of $f$ is $(x-y)^2$. Hence the linear form $ux+vy$ must be proportional to $x-y$.  Therefore the
linear part of $a(ux+vy+w)$, namely $(2\alpha w+\beta)(ux+vy)$,
is also proportional to $x-y$.  This contradicts the fact that the linear part of $f$ is $x+y$, which is not proportional to $x-y$.  Thus $f$ is not additively special.

Now suppose that $f(x,y)=a\bigl(b(x)c(y)\bigr)$. Again $b$ and $c$ must be nonconstant.  Let $m=\deg a$, $r=\deg b$, $s=\deg c$. The highest-degree term of $a(b(x)c(y))$ has total degree $m(r+s)$, with nonzero leading monomial proportional to $x^{mr}y^{ms}$.  Since
$\deg f=2$ and $r,s\ge1$, we must have $m=1$ and $r=s=1$. Thus $a,b,c$ are all linear, and the quadratic homogeneous part of
$a(b(x)c(y))$ is a scalar multiple of $xy$.  But the quadratic homogeneous part of $f$ is $x^2-2xy+y^2$, which has nonzero $x^2$ and $y^2$ terms.  This is also impossible.
\end{proof}

\section{Concluding remarks}
\label{sec:amplification}

The proof above may be summarized in one sentence: a fixed local defect is replicated until it becomes macroscopic.  For every prime \(p_i\mid Q\), one has $f_Q(a,b)\equiv(a-b)^2\pmod{p_i}$. Thus the value of \(f_Q\) is confined to the squares in \(\mathbb F_{p_i}\), a set
of density $\delta_i=\frac{p_i+1}{2p_i}<1$. A single prime supplies only a fixed saving.  The decisive issue is how this one restriction gets replicated into many independent copies. 

There are two natural directions of amplification.  In the prime aspect, or \emph{horizontally}, one remains entirely over \(\mathbb Q\) and introduces as many rational primes \(p_i\) as the ambient scale permits.  The Chinese remainder theorem then converts a fixed local gain at one prime into a product gain over a squarefree modulus.  Since a modulus $Q$ of size at most $N$ can contain roughly \(\log N/\log\log N\) primes of logarithmic size, this mechanism naturally produces savings of the form $\exp\!\left(-c\frac{\log N}{\log\log N}\right)$. In the degree aspect, or \emph{vertically}, one fixes the rational primes and passes instead to number fields in which each of them splits into many prime ideals. 

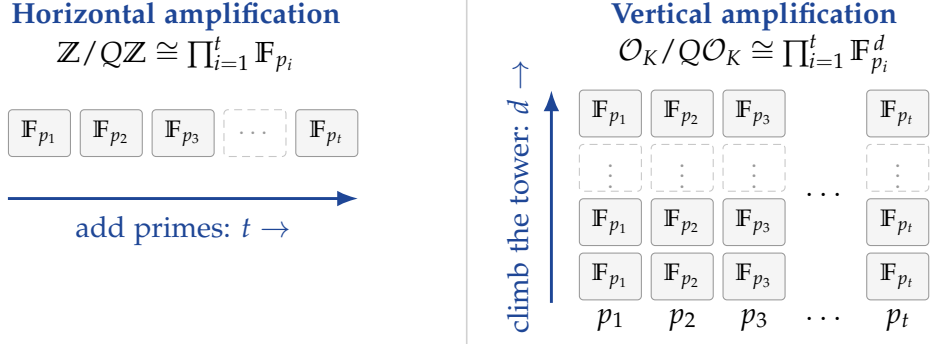
\begin{figure}[h]
\centering
\begin{tikzpicture}[
    >=Latex,
    cell/.style={
        draw=black!40,
        rounded corners=1.5pt,
        minimum width=0.82cm,
        minimum height=0.62cm,
        fill=black!4,
        inner sep=0pt,
        font=\scriptsize
    },
    ghost/.style={
        draw=black!22,
        densely dashed,
        rounded corners=1.5pt,
        minimum width=0.82cm,
        minimum height=0.62cm,
        fill=none,
        inner sep=0pt,
        font=\scriptsize,
        text=black!40
    },
    amp/.style={
        ->,
        RoyalBlue,
        line width=1.1pt
    },
    panellabel/.style={
        font=\bfseries\small,
        text=RoyalBlue
    },
    modulus/.style={
        font=\small
    }
]

\begin{scope}[shift={(0,0)}]
    \foreach \c/\lab in {
        0/{\F_{p_1}},
        1/{\F_{p_2}},
        2/{\F_{p_3}}
    }{
        \node[cell] (h-\c) at (\c*0.95,0.45) {$\lab$};
    }

    \node[ghost] (h-3) at (3*0.95,0.45) {$\cdots$};
    \node[cell]  (h-4) at (4*0.95,0.45) {$\F_{p_t}$};

    \draw[amp]
        ($(h-0.south west)+(0,-0.55)$)
        --
        ($(h-4.south east)+(0,-0.55)$)
        node[midway,below=2pt,font=\small,text=RoyalBlue]
        {add primes: \(t\rightarrow\)};

    \node[panellabel] at (1.9*0.95,2.0)
        {Horizontal amplification};

    \node[modulus] at (1.9*0.95,1.5)
        {$\Z/Q\Z\cong\prod_{i=1}^{t}\F_{p_i}$};
\end{scope}

\begin{scope}[shift={(7.55,0)}]

    \node[panellabel] at (1.9,2.0)
        {Vertical amplification};

    \node[modulus] at (1.9,1.5)
        {$\OK/Q\OK\cong\prod_{i=1}^{t}\F_{p_i}^{\,d}$};

    \foreach \c/\lab in {
        0/{p_1},
        1/{p_2},
        2/{p_3},
        4/{p_t}
    }{
        \node[cell]  (v-\c-0) at (\c*0.95,-1.45) {$\F_{\lab}$};
        \node[cell]  (v-\c-1) at (\c*0.95,-0.73) {$\F_{\lab}$};
        \node[ghost] (v-\c-2) at (\c*0.95,-0.01) {$\vdots$};
        \node[cell]  (v-\c-3) at (\c*0.95, 0.71) {$\F_{\lab}$};

        \node[modulus] at (\c*0.95,-2.02) {$\lab$};
    }

    \node[modulus] at (3*0.95,-0.37) {$\cdots$};
    \node[modulus] at (3*0.95,-2.02) {$\cdots$};

    \draw[amp]
        (-0.78,-1.80)
        --
        (-0.78,1.02)
        node[midway,above=4pt,sloped,font=\small,text=RoyalBlue]
        {climb the tower: \(d\rightarrow\)};
\end{scope}

\draw[black!18,line width=0.6pt]
    (5.65,-2.35) -- (5.65,2.25);

\end{tikzpicture}

\caption{Horizontal and vertical amplification of the same local defect.}
\label{fig:horizontal-vertical}
\end{figure}

The horizontal principle on its own can sometimes be surprisingly powerful. It was recently introduced and developed more systematically in the recent work of Croot, Mao, the author, Sheffer, and Yip~\cite{CMPSY2026}, where it became a new combinatorial large sieve for sets with bounded algebraic multiplicities. The local input in \cite{CMPSY2026} varies from problem to problem. Squares occupy only about half of the residue classes modulo an odd prime; the zero set of a binary quadratic form splits into isotropic lines modulo primes for which its discriminant is a square; and a norm form factors into linear forms modulo rational primes which split completely in the underlying number field. Passing to squarefree products packages these local restrictions into many compatible Chinese remainder theorem branches. If the set is then well distributed modulo the chosen squarefree modulus, then the proliferation of local branches forces many pairs or tuples satisfying the relevant congruence. A Sidon-type or bounded-representation hypothesis, however, places an upper bound on the number of such coincidences. 

Overall, these ideas from \cite{CMPSY2026} led to wide range of new results. For example, every Sidon subset of \(\{1^2,\ldots,N^2\}\) has size at most $N\exp\!\left(-c\frac{\log N}{\log\log N}\right)$. This gave the first super-polylogarithmic saving for a classical problem of Alon and Erd\H{o}s \cite{AlonErdos85}. Improving also on a classical estimate of Erd\H{o}s and Guy \cite{ErdosGuy70} from 1970, an appropriate two-dimensional version of this argument then showed that a subset of the grid \([N]^2\) with no repeated distance has size at most $N\exp\!\left(-c\frac{\log N}{\log\log N}\right)$. Furthermore, the same method gave new upper bounds for subsets of \([N]^2\) containing no isosceles triangle, another old problem of Erd\H{o}s, which was also recently popularized by \cite{PatternBoost}. Croot-Mao-Pohoata-Sheffer-Yip then also further developed an entropic version of this horizontal combinatorial sieve to further establish results for \(B_2[g]\)-sets in the squares. Perhaps most notably, using these additional entropy ideas, this method also yielded the first nontrivial bounds for \(B_3[g]\)-sets in the cubes and \(B_4[g]\)-sets in the fourth powers. Such problems are closely related to various classical questions in analytic number theory, see for example the survey of Maynard \cite{M26}.

Now, at the other extreme, for (vertical) amplification in the degree aspect, one makes a completely orthogonal move. Rather than staying within $\mathbb{Q}$, the number-field construction from Section \ref{sec:construction} moves instead in the ``vertical" direction. The primes
\(p_1,\ldots,p_t\), and hence the squarefree integer \(Q=p_1 \ldots p_t\), are fixed once and for all, and one instead starts climbing a tower of number fields in which each \(p_i\) splits completely.  In a degree-\(d\) field \(K\),
\[
        p_i\mathcal O_K
        =
        \prod_{j=1}^d\mathfrak p_{i,j},
        \qquad
        \mathcal O_K/\mathfrak p_{i,j}
        \cong
        \mathbb F_{p_i},
\]
and consequently
\[
        \mathcal O_K/Q\mathcal O_K
        \cong
        \prod_{i=1}^t\mathbb F_{p_i}^{\,d}.
\]
Each prime that supplied one coordinate over \(\Q\) now supplies \(d\)
independent ones, so the warm-up density $\rho=\delta_1 \ldots \delta_t$ becomes \(\rho^d\), as recorded in \eqref{eq:Omega-size}. The bounded root discriminant then supplies the complementary conversation of scales. For fixed \(X\), the Minkowski boxes satisfy
\(|B_K(X)|=\exp(\Theta(d))\).  After the archimedean packing losses, the
residue saving appears in \eqref{eq:theta-bound-main} as
\[
        |f_Q(P,P)|\le\theta^d|P|^2,\ \ \text{where}\ \ \theta=13D\rho<1.
\]
Since \(|P|=\exp(\Theta(d))\), the factor \(\theta^d\) is a fixed negative power of \(|P|\), thereby leading to a genuine power saving.

\medskip
{\noindent \bf{Further applications.}}
We expect that further combinations of the ideas above will similarly upgrade Ramanujan-type savings to genuine power savings in other extremal problems in combinatorics or number theory. 

For example, in a forthcoming joint work \cite{LeePohoataZhu}, as a further proof of concept, we will record how a vertical amplification of the Croot-Mao-Pohoata-Sheffer-Yip distance sieve from the proof of \cite[Theorem 1.5]{CMPSY2026} can also be used to construct a simpler and more robust counterexample for the unit distance conjecture, in the following sense. 

\begin{theorem} There exists an absolute constant $c>0$ and a set of $n$ points $P \subset \mathbb{R}^{2}$ for which the following statements (simultaneously) hold:
\begin{enumerate}
    \item some distance determined by $P$ occurs \(\geq n^{1+c}\) times,
    \item every subset of $P$ of size \(\geq n^{1/2-c}\) determines a repeated distance,
    \item every subset of $P$ of size \(\geq n^{1-c}\) contains an isosceles triangle.
\end{enumerate}
\end{theorem}
The existence of a set of points with any single one of the last two features alone is new, with each instance giving polynomially improved estimates for old problems of Erd\H{o}s. In fact, the existence of a set solely satisfying property (3) already confirms a conjecture of Erd\H{o}s from 1980. See for example \cite[page 110]{EG80} or \cite{BloomErdos1207}. We hope that the list of new ``horizontal" results from \cite{CMPSY2026} will inspire other such constructions.

\medskip
\section*{Acknowledgments}
Part of this research was conduced while the author was attending the conference ``Combinatorics and Geometry in Mytilene", held in Greece from May 25th to May 29th. The author would like to thank the organizers (Karim Adiprasito, Enis Kaya, Evrydiki Nestoridi, Stavros Papakadis, Vasiliki Petrotou and Christos Tatakis) for providing wonderful working conditions. The author would also like thank Thomas Bloom, Ernie Croot, Junzhe Mao, Oliver Roche-Newton, Will Sawin, Adam Sheffer, Jozsef Solymosi, Kyle Yip, and Daniel Zhu for helpful discussions.

We would also like to acknowledge the usage of AI in preparing and proofreading this manuscript. All the mathematical ideas in this work are human generated.

\end{document}